\documentclass{gtpart}     % Basic GT/GTM/AGT style
\title{Centric linking systems of locally finite groups}

\author{R\'emi Molinier}
\givenname{R\'emi}
\surname{Molinier}
\address{}
\email{molinier@ksu.edu}
\urladdr{}

\arxivreference{}
\arxivpassword{}

\volumenumber{}
\issuenumber{}
\publicationyear{}
\papernumber{}
\startpage{}
\endpage{}
\doi{}
\MR{}
\Zbl{}
\received{}
\revised{}
\accepted{}
\published{}
\publishedonline{}
\proposed{}
\seconded{}
\corresponding{}
\editor{}
\version{}

\newtheorem{thm}{Theorem}[section]    % Standard theorem environment
\newtheorem{lem}[thm]{Lemma}          % Lemma environment with numbering
\newtheorem{prop}[thm]{Proposition}

\newtheorem{thm*}{Theorem}
\newtheorem{prop*}{Proposition}

\theoremstyle{definition}
\newtheorem{defi}[thm]{Definition}

\newtheorem{step}{Step}

\usepackage{fullpage}
\usepackage{enumerate}
\usepackage[all]{xy}

\numberwithin{equation}{subsection}

\newcommand{\F}{\mathbb{F}}
\newcommand{\Zb}{\mathbb{Z}}
\newcommand{\Zp}{\mathbb{Z}_{(p)}}

\newcommand{\Li}{\mathcal{L}}
\newcommand{\T}{\mathcal{T}}

\newcommand{\Cc}{\mathcal{C}}

\newcommand{\op}{\text{op}}

\newcommand{\Top}{\text{Top}}
\newcommand{\h}{\mathcal{H}}

\newcommand{\Or}{\mathcal{O}}

\newcommand{\pcompl}[1]{#1^{\wedge}_p}
\newcommand{\limproj}[1]{\lim\limits_{\substack{\longleftarrow \\ #1}}}
\DeclareMathOperator*{\hocolim}{hocolim}

\newcommand{\Ob}{\text{Ob}}
\newcommand{\Map}{\text{Map}}

\newcommand{\Aut}{\text{Aut}}

\newcommand{\Mor}{\text{Mor}}

\newcommand{\incl}{\text{incl}}

\newcommand{\Ker}{\text{Ker}\,}

\newcommand{\Syl}{\text{Syl}}

\newcommand{\Ab}{\text{$\mathcal{A}$b}}
\begin{document}

\begin{abstract}    % type your abstract below
These notes are defining the notion of centric linking system for a locally finite group If a locally finite group $G$ has countable Sylow $p$-subgroups, we prove that, with a countable condition on the set of intersections, the $p$-completion of its classifying space is homotopy equivalent to the $p$-completion of the nerve of its centric linking system.
\end{abstract}

\maketitle

The notion of centric linking system of finite group was first introduced by Broto Levi and Oliver \cite{BLO1} to study the $p$-completion of classifying space of finite groups. It was the main tool in the proof of the Martino-Priddy Conjecture by Oliver \cite{O1,O2}. Later, Broto, Levi and Oliver define the notion of centric linking system associated to a saturated fusion system over a finite $p$-group \cite{BLO2} or a discrete $p$-toral group \cite{BLO3} to construct classifying spaces for fusion systems and develop the homotopy theory of fusion systems. In \cite{BLO3} they also generalize centric linking systems of finite groups to centric linking systems of locally finite groups with discrete $p$-toral Sylow $p$-subgroups. They also prove \cite[Theorem 8.7]{BLO3} an important property of centric linking systems : given a locally finite group $G$ with  discrete $p$-toral Sylow $p$-subgroups and satisfying some technical stabilization condition on centralizers, then the $p$-completion of the nerve of its centric linking systems has the homotopy type of the $p$-completion of the classifying space of $G$. 

On the other hand, Chermak and Gonzalez \cite{CG1}, using the language of localities are considering fusion systems over countable $p$-groups.This allows to consider fusion systems of a much more larger class of groups which contains in particular algebraic groups over the algebraic closure of $\F_p$. The groups they are considering are countable locally finite groups with a finite dimensionality condition on a certain poset of $p$-subgroups. This condition guarantee in particular the existence of Sylow $p$-subgroups and allow a study of the $p$-local structure of these groups.
%that they called lim-finite groups (see Definition \ref{d:lim-finite}).  This class is stable by subgroup, homomorphic image and extension and include, for example, algebraic group over the algebraic closure of $\F_p$. Thus it gives a nice class of groups to consider.

Here we generalize the notion of centric linking system to any locally finite groups. We are in particular interested in the case of localy finite group with countable Sylow $p$-subgroups. We prove in Theorem \ref{t:li=BG} that, for a locally finite group $G$ with countable Sylow $p$-subgroups, with a small countability condition on the poset of intersections of Sylows of $G$, the $p$-completion of the nerve of the centric linking system has the homotopy type of the classifying space of $G$.  This generalize the previous result of Broto, Levi and Oliver and can be the starting point of a homotopy theory of discrete localities developed in \cite{CG1}.
The surprising part of this result is that we get some information on these $p$-completions even if we do not know that the spaces we are considering are $p$-good. 

\section{Sylow $p$-subgroups}

In this paper, a \emph{$p$-group} is a locally finite group where every element of $P$ has finite order a power of $p$.

\begin{defi}
Let $G$ be a group let $S\leq G$ be a $p$-subgroup. We say that $S$ is a \emph{Sylow $p$-subgroup} of $G$ if

\begin{enumerate}[(i)]
\item $S$ is maximal in the poset of $p$-subgroups of $G$.
\item every $p$-subgroup of $G$ is conjugate to a subgroup of $S$; and
\end{enumerate}
We denote by $\Syl_p(G)$ the set of all Sylow $p$-subgroups of $G$.
\end{defi}

\begin{lem}
Let $G$ be a group with $\Syl_p(G)\neq \emptyset$.
\begin{enumerate}[(a)]
\item Any two elements in $\Syl_p(G)$ are conjugate.
\item Let $S$ be a $p$-subgroup of $G$ maximal in the poset of $p$-subgroups of $G$, then $S\in\Syl_p(G)$.
\end{enumerate}
\end{lem}

\begin{proof}
Let $S$ be a $p$-subgroup of $G$ maximal in poset of $p$-subgroups of $G$ and $S'\in\Syl_p(G)$. 
Since $S'$ is a Sylow $p$-subgroup of $G$, there is $g\in G$ such that $S^g\leq S'$. Assume that $S^g<S'$. Then $(S')^{g^{-1}}$ is a $p$-subgroup of $G$ which contains strictly $S$ and this contradicts the maximality of $S$. Thus $S^g=S'$ and this prove (a) and (b).
\end{proof}

For $G$ a group such that $\Syl_p(G)$ is non-empty, we denote by $\Omega_p(G)$ the collection of all subgroups of $G$ which are intersections of Sylow $p$-subgroups of $G$. Since $\Syl_p(G)$ is closed by conjugation in $G$, $\Omega_p(G)$ is also closed by conjugation in $G$. If $S\in\Syl_p(G)$ we will also define $\Omega_S(G)=S\cap\Omega_p(G)$ the collection of subgroups of $S$ which are in $\Omega_p(G)$.

\begin{defi}
Let $G$ be a group with $\Syl_p(G)\neq\emptyset$.
For $P$ a $p$-subgroup of $G$ we define $P^\circ\in\Omega_p(G)$ as the intersection of all Sylow $p$-subgroups of $G$ containing $P$.
\end{defi}

For $G$ a group and $P,Q$ two subgroups of $G$ we denote by $N_G(P,Q)$ the set of elements of $g\in G$ such that $P^g\leq Q$.
\begin{prop}\label{p:circ prop}
Let $G$ be a group with $\Syl_p(G)\neq\emptyset$ and $P, Q$ be two $p$-subgroups of $G$.
\begin{enumerate}[(a)]
\item $P\leq P^\circ$ and, if $P\leq Q$, then $P^\circ\leq Q^\circ$.
\item If $P\in\Omega_p(G)$, $P^\circ=P$.
\item $N_G(P,Q)\subseteq N_G(P^\circ,Q^\circ)$.
\item If $Q\in\Omega_p(G)$ then $N_G(P^\circ,Q)=N_G(P,Q)$.
\end{enumerate}
\end{prop}

\begin{proof}
(a) follows from the definition of $(-)^\circ$.
(b) is a direct consequence of the definition of $\Omega_p(G)$.
To prove (c) let $g\in N_G(P,Q)$. By a direct calculation we have $(P^\circ)^g=(P^g)^\circ$. Thus, by (a),
\[(P^\circ)^g= (P^g)^\circ\leq Q^\circ\]
and $g\in N_G(^\circ,Q^\circ)$.
Finally, (d) follows from (a) and (c).
\end{proof} 

\section{Centric linking systems}
In this section, we will mostly work with locally finite groups even though some definitions make sens for any group or at least torsion groups.

For $G$ a locally finite group, we denote by $\T_p(G)$ the \textit{transporter system} of $G$, this is the category with set of objects the collection of $p$-subgroups of $G$ and for morphisms
\[\Mor_{\T_p(G)}(P,Q)=N_G(P,Q):=\left\{g\in G\mid P^g\leq Q\right\}.\]

\begin{defi}
Let $G$ be a locally finite group. A $p$-subgroup $P\leq G$ is \emph{p-centric} if $C_G(P)/Z(P)$ has no element of order prime to $p$.
We denote by $\Omega_p^c(G)\subseteq \Omega_p(G)$ the subposet of $G$ consisting of all subgroups in $\Omega_p(G)$ which are $p$-centric. Also $\T_p^c(G)\subseteq\T_p(G)$ will denote the full subcategory of $\T_p(G)$ with set of objects the collection of $p$-centric subgroups of $G$.
\end{defi}

For $G$ a locally finite group, we define $O^p(G) \unlhd G$ the subgroup of $G$ generated by all elements of order prime to $p$.  

\begin{lem}
Let $G$ be a locally finite group and $P$ a $p$-subgroup of $G$. The following are equivalent.
\begin{enumerate}[(i)]
\item $P$ is $p$-centric.
\item $C_G(P)=Z(P)\times O^p(C_G(P))$ and all elements of $O^p(G)$ have order prime to $p$.
\end{enumerate}
\end{lem}

\begin{proof} 
the proof is the same as in \cite[Proposition 8.5]{BLO3}.
\end{proof}

\begin{defi}
Let $G$ be a locally finite group. 
The \emph{centric linking system} of $G$ is the category $\Li_p^c(G)$ whose set of objects is the collection of all the $p$-centric subgroups of $G$, and where
\[
\Mor_{\Li_p^c(G)}(P,Q)=N_G(P,Q)/O^p(C_G(P)).
\]
If $S\in\Syl_p(G)$, the equivalent full subcategory $\Li_S^c(G)\subseteq \Li_p^c(G)$ with objects the subgroups of $S$ which are $p$-centrics is called the \emph{centric linking system of $G$ over $S$}.
\end{defi}

\begin{lem}\label{l:functor p-kernel}
Let	$\Psi:\Cc\rightarrow\Cc'$ be a functor between small categories.
Assume the following:
\begin{enumerate}[(i)]
\item $\Psi$ is bijective on isomorphism classes of objects and is surjective on morphism sets;
\item for each object $c\in \Cc$, the subgroup
\[K(c)=\Ker\left[\Aut_\Cc(c)\rightarrow\Aut_{\Cc'}(\Psi(c))\right]\]
is a $p$-group %(maybe not finite)
; and
\item for each pair of objects $c$ and $d$, and each $f,g:c\rightarrow d$ in $\Cc$, $\Psi(f)=\Psi(g)$ if and only if there is some $\sigma\in K(c)$ such that $g=f\circ\sigma$ (i.e. $\Mor_{\Cc'}(\Psi(c),\Psi(d))\cong\Mor_\Cc(c,d)/K(c)$).
Then for any functor $F:\Cc'\rightarrow \Top$, the induced map
\[\hocolim_{\Cc'}(F)\rightarrow\hocolim_{\Cc}(F\circ\Psi)\]
is an $\F_p$-homology equivalence, and hence induces a homotopy equivalence between the $p$-completions. %Also for any two functor $T:\Cc'^{op}
\end{enumerate}
\end{lem}

\begin{proof}
This is \cite[Lemma 1.3]{BLO1} except that we are just asking $K(c)$ to be a $p$-group instead of a finite $p$-group.   But this suffices to ensure that coinvariants preserve exact sequences of $\Zp[K(c)]$-modules, which is the only way the condition on $K(c)$ is used in the proof of \cite[Lemma 1.3]{BLO1}.
\end{proof}

In particular, when $G$ is locally finite, the canonical projection functor $\T^c_p(G)\rightarrow \Li_p^c(G)$ satisfies all of the hypotheses of Lemma \ref{l:functor p-kernel}, Hence, the induced map give an homotopy equivalence
\begin{equation}\label{T=L}
\pcompl{|\T_p^c(G)|}\cong\pcompl{|\Li_p^c(G)|}
\end{equation}
\section{Higher limits over orbit categories}

\begin{defi}
Let $G$ be a group and $\h$ a collection of subgroups of $G$. 
The \emph{orbit category} of $G$ over $\h$ is the category $\Or_\h(G)$ with set of objects $\h$ and morphisms
\[
\Mor_{\Or_\h(G)}(H,H')=H'\setminus N_G(H,H')\cong \Map_G(G/H,G/H').
\] 
When $1\in\h$, for $M\in\Zb[G]$-module, we define
\[
\Lambda_\h^*(G;M)=\limproj{\Or_\h(G)}{}^{\hspace{-0.6em}*}(F_M),
\]
where $F_M:\Or_\h(G)\rightarrow\Ab$ is the functor defined by setting $F_M(H)=0$ if $H\neq 1$ and $F_M(1)=M$.
\end{defi}

By Proposition \ref{p:circ prop}, if $G$ is a group with $\Syl_p(G)\neq\emptyset$ we have a functor $(-)^\circ:\Or_p(G)\rightarrow \Or_{\Omega_p(G)}(G)$ and we have the following adjunction.
\begin{lem}\label{l:adjunction}
Let $G$ be a group with $\Syl_p(G)\neq \emptyset$.
The two functors
\[\xymatrix{\Or_{\Omega_p(G)}(G)  \ar@/^/[rr]^-{\incl}& &\ar@/^/[ll]^-{(-)^\circ} \Or_p(G)}\]
are adjoint.
\end{lem}

\begin{proof}
This is a direct consequence of Proposition \ref{p:circ prop}.
\end{proof}

For $G$ a group we denote by $S_p(G)$ the collection of all $p$-subgroups of $G$, $\Or_p(G)=\Or_{S_p(G)}(G)$ the associated orbit category of $G$ and, for $M$ a $\Zb[G]$-module, $\Lambda_p^*(G;M)=\Lambda_{S_p(G)}^*(G;M)$.

\begin{lem}[cf. {\cite[Lemma 5.10]{BLO3}}]\label{l:5.10}
Let $G$ be a group and $P$ be a $p$-subgroup of $G$.
Let $\Phi:\Or_p(G)^{\op}\rightarrow \Ab$ be a functor such that $\Phi(P)=0$ except when $P$ is $G$-conjugate to $Q$. set $\Phi':\Or_p(N_G(Q)/Q)\rightarrow \Ab$ to be the functor $\Phi'(P/Q)=\Phi(P)$. Then 
\[
\limproj{\Or_p(G)}{}^{\hspace{-0.5em}*}(\Phi)\cong\limproj{\Or_p(N_G(Q)/Q)}{}^{\hspace{-1.7em}*}(\Phi')\cong \Lambda_p^*(N_G(Q)/Q;\Phi(Q)).
\]
\end{lem}

\begin{proof}
This is a direct application of \cite[Proposition 5.3]{BLO3} with $\Cc=\Or_p(G)$, $\Gamma=N_G(Q)/Q$ and $\h=S_p(G)$.
\end{proof}

\begin{lem}[cf. {\cite[Proposition 5.12]{BLO3}}]\label{l:5.12}
Let $G$ be a locally finite group. Assume there is a countable $p$-subgroup $S\leq G$ such that every $p$-subgroup of $G$ is conjugate to a subgroup of $S$. Fix a $\Zb[G]$-module $M$ and assume that there exist a finite subgroup $H\leq G$ such that $\Lambda_p^*(K;M)=0$ for all subgroup $K\leq G$ containing $H$. Then $\Lambda_p^*(G;M)=0$.
In particular, $\Lambda_p^*(G;M)=0$ if $M$ is a $\Zp[G]$-module and the kernel of the action of $G$ on $M$ contains an element of order $p$.
\end{lem}

\begin{proof}
The proof is exactly the same as the proof of \cite[Proposition 5.12]{BLO3}. 
Indeed, they prove the result for $S$ discrete $p$-toral group but the only property of discrete $p$-toral groups they used is that $S$ is a increasing union of finite groups, which is also true for countable locally finite groups.
\end{proof}

\begin{lem}[{\cite[Lemma 5.11]{BLO3}}]\label{l:5.11}
Let $\Cc$ be a small category and let $\Cc_1\subseteq \Cc_2\subseteq \cdots$ be an increasing sequence of subcategories of $\Cc$ whose union is $\Cc$. Let $F:\Cc^{\op}\rightarrow \Ab$ be a functor such that for each $k$,
\[\limproj{i}{}^{\!1} \left( \limproj{\Cc_i}{}^{\!k} (F|_{\Cc_i})\right)=0.\]
Then the restrictions induce an isomorphism
\[\limproj{\Cc}{}^{\!k}(F)\cong \limproj{i} \left( \limproj{\Cc_i}{}^{\!k} (F|_{\Cc_i}) \right).\]
\end{lem}

\begin{lem}\label{l:O5 1.6}
let $G$ be a group. let $\h\subseteq \h'$ be collections of $p$-subgroups of $G$ closed by conjugations such that
\[ \text{for all $P,Q\in\h$, if $P\in\h'$ and $P\leq Q$ then $Q\in\h'$.}\]
 Let $F:\Or_\h(G)^{\op}\rightarrow \Ab$ be a functor and denote by $F|_{\Or_{\h'}(G)}$ the restriction of $F$ to $\Or_{\h'}(G)$ 
If for all $P\in\h\smallsetminus\h'$, $F(P)=0$, then 
\[\limproj{\Or_\h(G)}{}^{\hspace{-0.5em}*}(F)=\limproj{\Or_{\h'}(G)}{}^{\hspace{-0.5em}*}(F|_{\Or_{\h'}(G)}).\]
\end{lem}

\begin{proof}
The proof is exactly the same as the proof of \cite[Lemma 1.6(a)]{O5}.
\end{proof}
\section{$p$-completion of classifying spaces}

\begin{thm}\label{t:li=BG}
Let $G$ be a locally finite group with $\Syl_p(G)\neq\emptyset$ and $S\in\Syl_p(G)$.
Assume that $S$ is countable and that $\Omega_p(G)\smallsetminus\Omega_p^c(G)$ contains countably many conjugacy classes of subgroups of $G$.
Then,
\[\pcompl{|\Li_S^c(G)|}\simeq \pcompl{|\Li_p^c(G)|} \simeq \pcompl{BG}.\]
\end{thm}

\begin{proof}
The proof is based on the proof of \cite[Theorem 8.7]{BLO3}. We will write $\Omega:=\Omega_S(G)$ and $\Omega^c:=\Omega_S^c(G)$ for short.
The first isomorphic holds since the categories $\Li_S^c(G)$ and $\Li_p^c(G)$ are equivalent. it remains to prove the last isomorphic.
\begin{step}
Let $\Psi$ and $\Phi$ be the following functors from $\Or_p(G)$ to spaces:
\[\Psi(P)=G/P \qquad \text{and}\qquad \Phi(P)=EG\times_G \Psi(P)\cong EG/P\cong BP.\]
Then, for any full subcategory $\Cc\subseteq\Or_p(G)$,
\[\hocolim_{\Cc}(\Psi)=\left.\left(\coprod_{n=0}^{\infty} \coprod_{G/P_0\rightarrow \cdots\rightarrow G/P_n} G/P_0\times \Delta^n\right) \middle/ \sim\right.\]
is the nerve of the category whose object are the cosets $gP$ for all $P\in\Ob(\Cc)$, and with a unique morphism $gP\rightarrow hQ$ exactly when $P^g\leq Q^h$. 
The category $\Or_{\Omega}(G)$ has an initial object (the intersection of all the Sylow $p$-subgroup of $G$). Thus $\hocolim_{\Or_\Omega(G)}(\Psi)$ is contractible. Since the Borel construction is an homotopy colimit, it commutes with other homotopy colimit and we have:
\[\hocolim_{\Or_\Omega(G)}(\Phi)\cong EG\times_G\left(\hocolim_{\Or_\Omega(G)}(\Psi)\right)\simeq BG.\] 
\end{step}
\begin{step}
For $Q\in\Omega\smallsetminus \Omega^c$ and $i\geq 0$, we define the functor $F_i^{[Q]}:\Or_p(G)^\op\rightarrow \Ab$ as follows
\[
F_i^{[Q]}(P)=
\left\lbrace
\begin{array}{ll}
H^i(BP,\F_p) & \text{if $P$ is $G$-conjugate to $Q$}\\ 
0 & \text{otherwise.}
\end{array}
\right.
\]
$C_G(Q)Q/Q\leq\Aut_{\Or_p(G)}(Q)=N_G(Q)/Q$ acts trivially on $F_i^{[Q]}(Q)$. Moreover, since $Q$ is not $p$-centric, $C_G(Q)Q/Q\cong C_G(Q)/Z(Q)$ contains an element of order $p$. Hence, by Lemma \ref{l:5.10} and Lemma \ref{l:5.12},
\[
\limproj{\Or_p(G)}{}^{\hspace{-0.5em}*}(F_i^{[Q]})\cong \Lambda^*\left(N_G(Q)/Q;F_i^{[Q]}\right)=0 \quad \text{for all }i,
\]
Therefore, by Lemma \ref{l:adjunction},
\begin{equation}\label{2}
\limproj{\Or_\Omega(G)}{}^{\hspace{-0.6em}*}(F_i^{[Q]})\cong \limproj{\Or_p(G)}{}^{\hspace{-0.6em}*}(F_i^{[Q]})=0
\end{equation}

\end{step}
\begin{step}
%Set $\Or_\Omega^c(G)$ be the full subcategory of $\Or_\Omega(G)$ with set of object the collection of all $p$-centric subgroups of $G$ which lies in $\Omega$. 
Let 
\[\Or_\Omega^c(G)=\Or_0\subseteq\Or_1\subseteq \cdots \subseteq \Or_\Omega(G)\]
be a sequence of full subcategories of $\Or_\Omega(G)$ such that $\bigcup_{r\geq 0} \Or=\Or_\Omega(G)$ and for all $r\geq 0$,  $\Ob(\Or_{r+1})\smallsetminus \Ob(\Or_{r})=\{Q_{r+1}\}^G$ is the $G$-conjugation class of a subgroup $Q_{r+1}\in\Omega$ and such that for all $P\in \Ob(\Or_r)$ and $P'\in\Omega$ with $P\leq P'\leq S$ then $P'\in \Ob(\Or_r)$.
For $r\geq 0$ define $F_{i,r}:\Or_{r+1}^\op\rightarrow \Ab$ define by 
\[F_{i,r}(P)=
\left\{
\begin{array}{ll}
F_i(P) & \text{if } P\in \Ob(\Or_r),\\
0 & else
\end{array}
\right.\]
By Lemma \ref{l:O5 1.6},

For all $r\geq 0$, 
\[\Ker\left[F_{i,r+1}|_{\Or_{r+1}}\twoheadrightarrow F_{i,r}\right]=F_i^{[Q_{r+1}]}|_{\Or_{r+1}}\]
and, by \eqref{2} and Lemma \ref{l:O5 1.6}, the higher limits of this functor vanish. Thus
\begin{equation}\label{3.1}
\limproj{\Or_{r+1}}{}^{\hspace{-0.4em}*}(F_{i,{r+1}}|_{\Or_{r+1}})\cong\limproj{\Or_{r+1}}{}^{\hspace{-0.4em}*}(F_{i,r})\cong\limproj{\Or_{r}}{}^{\hspace{-0.3em}*}(F_{i,r}|_{\Or_{r}})
\end{equation}
where the last isomorphisms follow by Lemma \ref{l:O5 1.6}.
Notice that for all $r\geq 0$, $F_{i,r}|_{\Or_{r}}=F_i|_{\Or_{r}}$ and that \eqref{3.1} implies that for all $r\geq 0$,
\[\limproj{\Or_{r}}{}^{\hspace{-0.2em}*}(F_{i,r}|_{\Or_{r}})\cong\limproj{\Or_{0}}{}^{\hspace{-0.2em}*}(F_{i}|_{\Or_{0}})= \limproj{\Or_\Omega^c(G)}{}^{\hspace{-0.6em}*}(F_{i}|_{\Or_\Omega^c(G)}).\]
We can then apply Lemma \ref{l:5.11} (the hypothesis on $\limproj{}{}^{\hspace{-0.1em}1}$ can be easily check by a direct calculation on the chain level) to get
\[\limproj{\Or_\Omega(G)}{}^{\hspace{-0.6em}*}(F_{i}|_{\Or_\Omega(G)})\cong \limproj{\Or_\Omega^c(G)}{}^{\hspace{-0.6em}*}(F_{i}|_{\Or_\Omega^c(G)})\]

The spectral sequence for cohomology of a homotopy colimit (\cite[XII.4.5]{BK}) now implies that the inclusion $\Or_\Omega^c(G)\subseteq\Or_\Omega(G)$ induces a mod $p$ homology isomorphism of homotopy colimits of $\Phi$ and hence a homotopy equivalence
\begin{equation}\label{3.2}
\pcompl{(\hocolim_{\Or_\Omega^c(G)}(\Phi))}\simeq 
\pcompl{(\hocolim_{\Or_\Omega(G)}(\Phi))}.
\end{equation}
Also, th adjunction of Lemma \ref{l:adjunction} restrict to an adjunction between $\Or_\Omega^c(G)$ and $\Or_p^c(G)$, and hence induces a homotopy equivalence
\begin{equation}\label{3.3}
\pcompl{(\hocolim_{\Or_\Omega^c(G)}(\Phi))}\simeq 
\pcompl{(\hocolim_{\Or_p^c(G)}(\Phi))}.
\end{equation}
\end{step}
\begin{step}
Now, by exactly the same argument as in \cite[Lemma 1.2]{BLO1} we have 
\begin{equation}\label{4.1}
\hocolim_{\Or_p^c(G)}(\Phi)\cong EG\times_G \left(\hocolim_{\Or_p^c(G)}(\Psi)\right) \cong |\T^c_p(G)|.
\end{equation}
\end{step}
Finally, by \eqref{2}, \eqref{3.1}, \eqref{3.2}, \eqref{3.3}, \eqref{4.1} and \eqref{T=L}
\[\pcompl{|\Li_p^c(G)|}\cong \pcompl{|\T_p^c(G)|}\cong \pcompl{BG}.\]
This ends the proof of Theorem \ref{t:li=BG}.
\end{proof}

\section{Particular cases}

Theorem \ref{t:li=BG} works for a very large class of groups. Here are some classical classes of groups which satisfy the hypothesis of Theorem \ref{t:li=BG}.

\begin{defi}
A \emph{discrete $p$-toral group} is a group $P$ with a normal subgroup $P_0\unlhd P$ such that
\begin{enumerate}[(a)]
\item $P$ is isomorphic to a finite product of copies of $\Z/p^\infty:=\bigcup_{n\geq 1} \Z/p\Z$; and
\item $P/P_0$ is a finite $p$-group.
\end{enumerate} 
\end{defi}

Theorem \ref{t:li=BG} give a generalization of the second part of \cite[Theorem 8.7]{BLO3} where they work with locally finite groups with disrete $p$-toral Sylow $p$-subgroups but with a condition of stabilization on centralizers. With Theorem \ref{t:li=BG}, we can also get rid of this condition on the centralizers if we require the group to be countable.  

\begin{lem}\label{l:discrete p-toral}
Let $G$ is a locally finite group.
Assume that
\begin{enumerate}[(a)]
\item each $p$-subgroups of $G$ is a discrete $p$-toral group; and
\item $G$ is countable.
\end{enumerate}
Then $G$ satisfies the hypotheses of Theorem \ref{t:li=BG}.
\end{lem}

\begin{proof} 
By \cite[Proposition 1.2]{BLO3} $S$ is artinian. In particular, $\Omega_S(G)$ is the set of intersection of $S$ and a finite collection of $G$-conjugates of $S$. Since $G$ is countable, $\Omega_S(G)$ is countable and $G$ satisfies the hypotheses of Theorem \ref{t:li=BG}.
\end{proof}

Moreover, Theorem \ref{t:li=BG} cover also countable locally finite groups which satisfies a condition of "finite dimensionality" which is central in \cite{CG1}. For $G$ a group and $H$ a subgroup of $G$ we denote by $\Omega^{\text{fin}}_H(G)$ the set of finite intersections of $G$-conjugate of $H$.

%\begin{defi}\label{d:lim-finite} For $G$ a group and $H$ a subgroup of $G$ we denote by $\Omega^{\text{fin}}_H(G)$ the set of finite intersections of $G$-conjugate of $H$.
%A group $G$ is \emph{lim-finite} if $G$ is locally finite and there are an increasing sequence $\{G_n\}_{n\geq 0}$ of finite subgroups of $G$ and a $p$-subgroup $S\leq G$ such that:
%\begin{enumerate}[(a)]
%\item $G=\bigcup_{n\geq 0} G_n$;
%\item $S$ is maximal in the poset of $p$-subgroups of $G$; and
%\item The supremum of the lengths of chains of proper inclusion in $\Omega^{\text{fin}}_S(G)$ exists and is finite.
%\end{enumerate}
%\end{defi}

\begin{lem}\label{l:lim-finite}
Let $G$ be a locally finite group.
Assume that
\begin{enumerate}[(a)]
\item $G$ is countable,
\item The supremum of the lengths of chains of proper inclusion in $\Omega^{\text{fin}}_S(G)$ exists and is finite.
\end{enumerate}
Then $S\in \Syl_p(G)$ and $G$ satisfies the hypotheses of Theorem \ref{t:li=BG}.
\end{lem}

\begin{proof}
$S$ is a Sylow $p$-subgroup of $G$ by \cite[Proposition 3.8]{CG1} apply to the locality $(G,\Delta,S)$ for $\Delta$ the collection of all subgroup of $S$. By (b), it easy to see that $\Omega^{\text{fin}}_S(G)$ and then, since $G$ is coountable by (a), $\Omega^{\text{fin}}_S(G)$ is countable. Hence, $G$ satisfies the hypotheses of Theorem \ref{t:li=BG}.
\end{proof}

Gonzalez and Chermak proved, using the Chevalet commutator formula, that an algebraic group over the algebraic closure of $\F_p$ satisfies the hypotheses of \ref{l:lim-finite}. In particular, any algebraic group over the algebraic closure of $\F_p$ satisfies the hypotheses of Theorem \ref{t:li=BG}.
%%%%%%%%%%%%%%%%%%%%   Start of main body of article

%%%%%%%%%%%%%%%%%%%%   End of main body of article
%
%                             References
%
%   BiBTeX users uncomment the following line:

\bibliography{biblio}{}

\begin{thebibliography}{BLO03b}

\bibitem[BK]{BK}
Aldridge~K. Bousfield and Daniel~M. Kan.
\newblock {\em Homotopy limits, completions and localizations}.
\newblock Lecture Notes in Mathematics, Vol. 304. Springer-Verlag, Berlin-New
  York, 1972.

\bibitem[BLO1]{BLO1}
Carles Broto, Ran Levi, and Bob Oliver.
\newblock Homotopy equivalences of {$p$}-completed classifying spaces of finite
  groups.
\newblock {\em Invent. Math.}, 151(3):611--664, 2003.

\bibitem[BLO2]{BLO2}
Carles Broto, Ran Levi, and Bob Oliver.
\newblock The homotopy theory of fusion systems.
\newblock {\em J. Amer. Math. Soc.}, 16(4):779--856, 2003.

\bibitem[BLO3]{BLO3}
Carles Broto, Ran Levi, and Bob Oliver.
\newblock Discrete models for the {$p$}-local homotopy theory of compact {L}ie
  groups and {$p$}-compact groups.
\newblock {\em Geom. Topol.}, 11:315--427, 2007.

\bibitem[CG1]{CG1}
Andrew Chermak and Alex Gonzalez.
\newblock Discrete localities $\text{I}$.
\newblock arXiv:1702.02595, 2017.

\bibitem[O1]{O1}
Bob Oliver.
\newblock Equivalences of classifying spaces completed at odd primes.
\newblock {\em Math. Proc. Cambridge Philos. Soc.}, 137(2):321--347, 2004.

\bibitem[O2]{O2}
Bob Oliver.
\newblock Equivalences of classifying spaces completed at the prime two.
\newblock {\em Mem. Amer. Math. Soc.}, 180(848):vi+102, 2006.

\bibitem[O5]{O5}
Bob Oliver.
\newblock Existence and uniqueness of linking systems: Chermak's proof via
  obstruction theory.
\newblock {\em Acta Math.}, 211(1):141--175, 2013.

\end{thebibliography}
\bibliographystyle{abstract}

%\bibliographystyle{gtart}
%%
%
%\begin{thebibliography}
%
%\end{thebibliography}

\end{document}